\input ./style/arxiv-vmsta.cfg
\documentclass[numbers,compress,v1.0.1]{vmsta}

\volume{2}
\issue{3}
\pubyear{2015}
\firstpage{297}
\lastpage{308}
\doi{10.15559/15-VMSTA40CNF}% Updated by VTEXPTS2LaTeX.exe, 29.09.2015
\startlocaldefs
\newcommand{\lleft}{\left}
\newcommand{\rright}{\right}
\newtheorem{thm}{Theorem}
\newtheorem{lemma}{Lemma}
\newtheorem{cor}{Corollary}
\newtheorem{pro}{Statement}%[section]

\theoremstyle{definition}
\newtheorem{remark}{Remark}
\newtheorem{defin}{Definition}
\newtheorem{exam}{Example}%[part]

\allowdisplaybreaks
\endlocaldefs

\begin{document}
\begin{frontmatter}

\title{Extreme residuals in regression model. Minimax approach}

\author[a]{\inits{A.}\fnm{Aleksander}\snm{Ivanov}\corref{cor1}}\email
{alexntuu@gmail.com}
\cortext[cor1]{Corresponding author.}
\author[b]{\inits{I.}\fnm{Ivan}\snm{Matsak}}\email{ivanmatsak@univ.kiev.ua}
\author[b]{\inits{S.}\fnm{Sergiy}\snm{Polotskiy}}\email
{sergiy.polotskiy@gmail.com}
\address[a]{National Technical University of Ukraine ``Kyiv Polytechnic
Institute'', 37~Peremogy Ave., Kyiv, 03056, Ukraine}
\address[b]{Taras Shevchenko National University of Kyiv, 4e
Academician Glushkov Ave., Kyiv, 03127, Ukraine}

\markboth{A. Ivanov et al.}{Extreme residuals in
regression model. Minimax approach}

\begin{abstract}
We obtain limit theorems for extreme residuals in linear regression
model in the case of minimax estimation of parameters.
\end{abstract}


\received{26 August 2015}% Updated by VTEXPTS2LaTeX.exe, 29.09.2015
%13:44
\revised{22 September 2015}% Updated by VTEXPTS2LaTeX.exe, 29.09.2015
%13:44
\accepted{22 September 2015}% Updated by VTEXPTS2LaTeX.exe, 29.09.2015
%13:44
\publishedonline{5 October 2015}
\end{frontmatter}

%s1 ###
\section{Introduction} Consider the model of linear regression
\begin{gather}
\label{def_regres} y_j = \sum_{i=1}^q
\theta_i x_{ji} + \epsilon_j, \quad j =
\overline{1,N},
\end{gather}
where $\theta= (\theta_1, \theta_2,\ldots,\theta_q)$
is an unknown parameter, $\epsilon_j$ are independent
identically distributed (i.i.d.) random variables (r.v.-s) with
distribution function (d.f.) $F(x)$, and $X=(x_{ji})$ is a
regression design matrix.

Let $\widehat{\theta} = (\widehat{\theta_1},\ldots,
\widehat{\theta_q})$ be the least squares estimator (LSE) of
$\theta$. Introduce the notation
\begin{align*}
&\widehat{y_j}=\sum_{i=1}^q \widehat{\theta_i}x_{ji}, \qquad \widehat{\epsilon_j}= y_j - \widehat{y_j},\quad j=\overline{1,N};\\
&Z_N = \max_{1\leq j \leq N}\epsilon_j, \qquad \widehat{Z_N} = \max_{1\leq j \leq N}\widehat{\epsilon_j},\\
&Z_N^* = \max_{1\leq j \leq N}|\epsilon_j|, \qquad \widehat{Z_N}^* = \max_{1\leq j \leq N}|\widehat{\epsilon_j}|.
\end{align*}

Asymptotic behavior of the r.v.-s $Z_N$, $Z_N^*$ is
studied in the theory of extreme values (see classical works by
Frechet \cite{fre}, Fisher and Tippet \cite{fit},
and Gnedenko \cite{gne} and monographs \cite{gal,llr}).
In the papers \cite{im13,im14}, it was shown that under
mild assumptions asymptotic properties of the r.v.-s $Z_N$,
$\widehat{Z_N}$, $Z_N^*$, and $\widehat{Z_N}^*$ are similar
in the cases of both finite variance and heavy tails of
observation errors $\epsilon_j$.

In the present paper, we study asymptotic properties of minimax estimator
(MME) of $\theta$ and maximal absolute residual. For
MME, we keep the same notation $\widehat{\theta}$.

\begin{defin}\label{def1}
A random variable $\widehat{\theta} =( \widehat{\theta_1},\ldots,
\widehat{\theta_q})$ is called MME for $\theta$ by the
observations (\ref{def_regres})
\begin{gather}
\label{def_MME} \widehat{\varDelta} = \varDelta(\widehat{\theta}) =\underset{\tau\in
\mathbb{R}^{q}} {\min}\varDelta(\tau),
\end{gather}
where
\begin{gather*}
\varDelta(\tau)= \underset{1 \leq j \leq N} {\max}\left|y_j - \sum
_{i=1}^{q}\tau_i x_{ji}\right|.
\end{gather*}

Denote
%$W_N = \underset{1 \leq j \leq N}{\min} \epsilon_j$
$W_N = {\min_{1 \leq j \leq N}} \epsilon_j$
and let $R_N = Z_N - W_N$ and $Q_N = \tfrac{Z_N +
W_N}{2}$ be the range and midrange of the sequence $\epsilon_j,\
j=\overline{1,N}$.
\end{defin}

The following statement shows essential difference in the behavior
of MME and LSE.

\begin{pro}\label{stat1}
\begin{enumerate}

\item[\emph{(i)}] If the model (\ref{def_regres}) contains a constant
term, namely,
$x_{j1} = 1$, $j = \overline{1,N}$, then almost
surely (a.s.)
\begin{gather}
\label{Delta_hat} \widehat{\varDelta} \leq\dfrac{R_N}{2}. \
\end{gather}
\item[\emph{(ii)}] If the model (\ref{def_regres}) has the form
\begin{gather}
\label{reg_form1_ii} y_j = \theta+ \epsilon_j, \quad j =
\overline{1,N} ,
\end{gather}
then a.s.
\begin{gather*}
\widehat{\varDelta} = \dfrac{R_N}{2}, \qquad\widehat{\theta} - \theta=
Q_N.
\end{gather*}
\end{enumerate}
\end{pro}

\begin{remark}\label{rem1}
From the point \textup{(ii)} of Statement~\ref{stat1} it follows that MME
$\widehat{\theta}$ is not consistent in the model
\eqref{reg_form1_ii} with some $\epsilon_j$ having all the moments
\textup{(}see Example~2\textup{)}.
\end{remark}

\begin{remark}\label{rem2}
The value $\widehat{\varDelta}$ can be represented as a solution of
the following linear programming problem \textup{(}LPP\textup{)}:
\begin{align}
\label{LPP} \widehat{\varDelta} &\,{=}\, \underset{\varDelta\in\mathcal{D}} {\min}\varDelta,\\
\mathcal{D} &\,{=}\, \Biggl\{ (\tau, \varDelta) \in\mathbb{R}^q\,{\times}\,\mathbb{R}_{+}: \left|y_j \,{-}\, \sum_{i=1}^{q}\tau_i x_{ji}\right| \,{\leq}\,\varDelta, \ j \,{=}\, \overline{1,N} \Biggr\}\notag\\
&\,{=}\,\Biggl\{ (\tau, \varDelta) \in \mathbb{R}^q  \,{\times}\,\mathbb{R}_{+}: \sum_{i=1}^{q}\tau_i x_{ji} \,{+}\, \varDelta \,{\geq}\, y_j, -\sum_{i=1}^{q}\tau _ix_{ji}\,{+}\, \varDelta \geq-y_j, \ j \,{=}\, \overline{1,N} \Biggr\}.\notag
\end{align}
\end{remark}

So, the problem (\ref{def_MME}) of determination of the values
$\widehat{\varDelta}$ and $\widehat{\theta}$ is reduced to solving LPP
(\ref{LPP}). The LPP can be efficiently solved
numerically by the simplex method; see \cite{elmt,zay}).
Investigation of asymptotic properties of maximal absolute residual
$\widehat{\varDelta}$ and MME
$\widehat{\theta}$ is quite
difficult in the case of general model (\ref{def_regres}). However,
under additional assumptions on regression experiment design and
observation errors $\epsilon_j$, it is possible to find the limiting
distribution of $\widehat{\varDelta}$, to prove the consistency of MME
$\widehat{\theta}$,
and even estimate the rate of convergence $\widehat{\theta}\to\theta$, $N\to\infty$.

%s2 ###
\section{The main theorems} First, we recall briefly some results
of extreme value theory. Let r.v.-s
$(\epsilon_j )$ have the d.f. $F(x) $. Assume that for some constants $b_n
> 0$ and $ a_n $, as $n\rightarrow\infty$,
%
%e1 ###
\begin{equation}
\label{f6} b_n (Z_n -a_n ) \stackrel{D}
\longrightarrow \zeta,
\end{equation}
and $\zeta$ has a nondegenerate d.f. $G(x) = \mathbb{P}( \zeta < x )$.
If assumption
(\ref{f6}) holds, then we say that d.f. $F$ belongs to the domain
of maximum attraction of the probability distribution
$G$ and write $F \in D(G) $.\vspace{10pt}

{\em If $F \in D(G) $, then $G$ must have just one of the following three
types of distributions \textup{\cite{gne,llr}:}\medskip

{Type I:}
\[
%\label{f5}
\varPhi_{\alpha} (x) =  %
\begin{cases}
0, &  x \leq0 , \\
\exp\big(-x^{-\alpha}\big), &  \alpha> 0, \  x > 0 ; \nonumber\\
\end{cases}
\]

{Type II:}
\[
\varPsi_{\alpha} (x) =
\begin{cases}
\exp\big(-(-x)^{\alpha}\big), &  \alpha> 0, \  x \leq0 , \\
1, &  x > 0 ;
\end{cases}
\]

{Type III:}
%
%e2 ###
\begin{eqnarray}
\label{f7}
\quad\varLambda(x) = \exp\bigl(-e^{-x}\bigr) , \, \- \infty< x
<\infty .
\end{eqnarray}
}\vspace{-10pt}

Necessary and sufficient conditions for convergence to each of
d.f.-s $ \varPhi_{\alpha}$, $\varPsi_{\alpha}$, $\varLambda$ are also well
known.\vadjust{\eject}

Suppose in the model (\ref{def_regres})
that:
\begin{description}
\item[(A1)] ($\epsilon_j$) are symmetric r.v.-s;
\item[(A2)] ($\epsilon_j$) satisfy relation (\ref{f6}), that is, $F \in
D(G) $
with normalizing constants $a_n$ and $b_n$, where $G$ is one
of the d.f.-s. $ \varPhi_{\alpha}$, $\varPsi_{\alpha}$, $\varLambda$
defined in (\ref{f7}).
\end{description}
Assume further that regression experiment design is organized as
follows:
\begin{align}
x_j &= (x_{j1},\ldots,x_{jq}) \in \{v_1, v_2,\ldots,v_k \},\quad v_l =(v_{l1}, \ldots, v_{lq}) \in \mathbb{R}^q,\nonumber\\
v_m &\neq v_l, \quad m\neq l;\label{x_j}
\end{align}
that is, $x_j$ take some fixed values only. Besides, suppose that
\begin{gather}
\label{x_j_V} x_j = V_l \quad\text{for} \ j \in
I_l, \ l = \overline{1,k},
\end{gather}
$\operatorname{card}(I_l) = n$, \xch{$I_m\cap I_l = \oslash$}{$I_m\bigcap I_l = \oslash$}, $m \neq l$,
 $N = kn$ is the sample size,
\begin{gather*}
V = \lleft(%
\begin{matrix}%{ccccc}
v_{11} & v_{12} & \ldots& v_{1q} \\
v_{21} & v_{22} & \ldots& v_{2q} \\
\ldots& \ldots& \ldots& \ldots\\
v_{k1} & v_{k2} & \ldots& v_{kq} \\
\end{matrix}
 \rright) .
\end{gather*}

%$Z_{nl} = \underset{j \in I_l}{\max} \epsilon_j$, $W_{nl} =
%\underset{j \in I_l}{\min} \epsilon_j$, $R_{nl} = Z_{nl} -
%W_{nl}$, $Q_{nl} = \dfrac{Z_{nl} + W_{nl}}{2}$.

\begin{thm}\label{thm1}

Under assumptions \textup{(A1)}, \textup{(A2)},
\eqref{x_j}, and \eqref{x_j_V},
\begin{gather}
\label{theor_1_conv} \varDelta_n = b_n(\widehat{\varDelta}
-a_n) \overset{D}\to\varDelta_0 , \quad n\to\infty,
\end{gather}
where
\begin{align}
\varDelta_0 &= \underset{u \in\mathcal{D}^*} {\max}L_0^*(u),\nonumber\\
L_0^*(u) &= \sum_{l=1}^{k}\bigl(u_l\zeta_l +u'_l\zeta'_l \bigr), \quad u = \bigl(u_1,\ldots,u_k,u'_1,\ldots,u'_k\bigr),\nonumber\\
\mathcal{D}^* &= \Biggl\{ u\geq0: \sum_{l=1}^{k}\bigl(u_l - u'_l\bigr)v_{li} = 0, \, \sum_{l=1}^{k} \bigl(u_l+ u'_l\bigr) = 1 ,\ i = \overline{1,q}\label{D_star} \Biggr\},
\end{align}
$\zeta_{l}$, $\zeta'_{l}$ , $l=\overline{1,k}$, are i.r.v.-s
having d.f. $G(x)$.
\end{thm}

For a number sequence $b_n \to\infty$ and random sequence $(\xi_n)$,
we will write\break $\xi_n \overset{P}= O(b_n^{-1})$ if
\begin{gather*}
\underset{n} {\sup}\mathbb{P}\bigl(b_n|\xi_n| > C \bigr) \to0
\quad\mbox{as} \ C\to\infty.
\end{gather*}

Assume that $k\geq q$ and there exists square submatrix $\widetilde{V} \subset V$ of order $q$
\begin{gather*}
\widetilde{V} = \lleft(%
\begin{matrix}%{ccc}
v_{l_{1}1} & \ldots& v_{l_{1}q} \\
\ldots& \ldots& \ldots\\
v_{l_{q}1} & \ldots& v_{l_{q}q} \\
\end{matrix}
 \rright),
\end{gather*}
such that
\begin{gather}
\label{V_determ} \det\widetilde{V} \neq0.
\end{gather}

\begin{thm}\label{thm2}
Assume that, under conditions of Theorem~\emph{\ref{thm1}}, $k \geq q$, assumption
\eqref{V_determ} holds and
\begin{gather}
\label{b_n_tend_infty} b_n \to\infty\quad as \ n\to\infty.
\end{gather}
Then MME $\widehat{\theta}$ is consistent, and
\begin{gather*}
\widehat{\theta}_i - \theta_i \overset{P} = O
\bigl(b_n^{-1}\bigr), \quad i=\overline{1,q} .
\end{gather*}
\end{thm}

\begin{exam}\label{exam1}

Let in the model of simple linear regression
\begin{gather}
\label{example_1} y_j=\theta_0 + \theta_1x_j
+ \epsilon_j, \quad j = \overline{1,N},
\end{gather}
$x_j = v$, $j=\overline{1,N}$, that is, $k=1$ and $q=2$.

Then such a model can be rewritten in the form
(\ref{reg_form1_ii}) with $\theta= \theta_0 + \theta_1 v$.
Clearly, the parameters $\theta_0$, $\theta_1$ cannot be defined
unambiguously here. So, it does not make sense to speak about
the consistency of MME $\widehat{\theta}$ when $k < q$.
\end{exam}

\begin{exam}\label{exam2}
Consider regression model
(\ref{reg_form1_ii}) with errors $\epsilon_j$ having the Laplace
density $f(x) = \frac{1}{2}e^{-|x|}$. For this distribution, the
famous von Mises condition is satisfied (\cite{llr}, p.~16) for the type
III distribution, that is, $F \in D(\varLambda)$. For symmetric $F \in
D(\varLambda), $ we have
\begin{gather*}
\underset{n \to\infty}\lim\mathbb{P}\{ 2b_nQ_n < x \}
= \dfrac{1}{1+e^{-x}}.
\end{gather*}
The limiting distribution is a logistic one (see \cite{m14},
p.~62). Using further well-known formulas for the type $\varLambda$
(\cite{m14}, p.~49) $a_n=F^{-1}(1-\frac{1}{n})$ and $ b_n=nf(a_n)$,
we find $ a_n = \ln\frac{n}{2}$ and $b_n = 1$.  From Statement
1 it follows now that MME $\widehat{\theta}$ is
not consistent. Thus, condition (\ref{b_n_tend_infty}) of Theorem~\ref{thm2}
cannot be weakened.
\end{exam}

The following lemma allows us to check condition
(\ref{b_n_tend_infty}).

\begin{lemma}\label{lemma1}
Let $F \in D(G)$. Then we have:
\begin{enumerate}
\item If $G = \varPhi_{\alpha}$, then
\begin{align*}
x_{F} &= \sup\bigl\{ x: F(x) <1\bigr\} = \infty, \qquad\gamma_n = F^{-1}\biggl(1-\frac{1}{n}\biggr)\to\infty, \\
b_n &= \gamma_n^{-1}\to 0 \quad\mbox{as }\ n\to\infty.
\end{align*}
Thus, \eqref{b_n_tend_infty} does not hold.
\item If $G = \varPsi_{\alpha}$, then
\begin{align*}
x_{F}<\infty,\quad 1 - F(x_{F}-x) = x^{\alpha}L(x),
\end{align*}
where $L(x)$ is a slowly varying (s.v.) function at
zero, and there exists s.v.\ at infinity function $L_1(x)$ such
that
\begin{align*}
b_n=(x_F - \gamma_n)^{-1} = n^{\alpha}L_1(n)\to\infty \quad\mbox{as }\ n\to\infty.
\end{align*}
So \eqref{b_n_tend_infty} is true.\vadjust{\eject}
\item If $G = \varLambda$, then
\begin{align*}
b_n=r(\gamma_n),\quad\mbox{where }\ r(x) = R^{\prime}(x), R(x) = -\ln(1-F(x)).
\end{align*}
Clearly, \eqref{b_n_tend_infty} holds if
\begin{align*}
x_F = \infty,\qquad r(x)\to\infty\quad\mbox{as }\ x\to\infty.
\end{align*}
\end{enumerate}
\end{lemma}

Similar results can be found in \cite{m14}, Corollary 2.7, pp.
44--45; see also \cite{gal,llr}.

Set\vspace{-6pt}
\begin{align*}
&Z_{nl} = \underset{j\in I_l}\max\epsilon_j, \qquad W_{nl} = \underset{j\in I_l}\min\epsilon_j\\
&R_{nl} = Z_{nl} - W_{nl}, \qquad Q_{nl} = \frac{Z_{nl}+ W_{nl}}{2}, \quad l = \overline{1,k}.
\end{align*}

It turns out that Theorems~\ref{thm1} and \ref{thm2} can be significantly simplified in
the case $k=q$.

\begin{thm}\label{t3}

Let for the model \eqref{def_regres}
conditions \eqref{x_j} and \eqref{x_j_V} be satisfied,
$k=q$, and a matrix $V$ satisfies condition \eqref{V_determ}. Then
we have:

\begin{enumerate}
\item[{\rm(i)}]\quad\vspace*{-23pt}
\begin{align}
\label{f15} \hat{\varDelta} &= \frac{1}{2} \max_{1\leq l \leq q}R_{nl},\\
\label{f16} \hat{\theta}_i -\theta_i &= \frac{\det VQ_{(i)}}{\det V} , \quad i= \overline{1,q},
\end{align}
where the matrix $VQ_{(i)}$ is obtained from V by replacement of
the $i$th column by the column $(Q_{n1}, \ldots, Q_{nq})^{T}$.

\item[{\rm(ii)}] If additionally conditions $(A_1), (A_2)$ are satisfied, then
%
%e3 ###
\begin{equation}
\label{f17} \lim_{n \rightarrow\infty} \mathbb{P}\bigl(2b_n (
\hat{\varDelta} -a_n ) < x\bigr) = \bigl(G\star G (x)
\bigr)^q ,
\end{equation}
where $ G\star G (x) = \int_{-\infty}^{\infty} G(x-y)dG(y), $
and for $i=\overline{1,q}$, as $n \rightarrow\infty$,
%
%e4 ###
\begin{equation}
\label{f18} 2b_n (\hat{\theta}_i -
\theta_i ) \stackrel{D} \longrightarrow \frac{\det V\zeta_{(i)}}{\det V} ,
\end{equation}
the matrix $V\zeta_{(i)}$ is obtained from the $V$ by the
replacement of the $i$th column by the column $(\zeta_1 -
\zeta'_1, \ldots, \zeta_q - \zeta'_q)^{T}$, where all the
r.v.-s $\zeta_i , \zeta'_i
$ are independent and have
d.f. $G$.
\end{enumerate}
\end{thm}

\begin{remark}\label{rem3}
Suppose that in the model \eqref{def_regres}, under assumptions
\eqref{x_j}, \eqref{x_j_V}, $k < q$, and there exists a
nondegenerate submatrix $\widetilde{V} \subset V$ of order $k$. Then
\begin{gather*}
\hat{\varDelta}\leq\frac{1}{2}\underset{1\leq l\leq k}\max R_{nl} \
\ \ a.s.
\end{gather*}
\end{remark}

\begin{remark}\label{rem4}
For standard LSE,
\begin{gather*}
\hat{\theta_i} - \theta_i \overset{P}= O
\bigl(n^{-1/2}\bigr);
\end{gather*}\vspace{-18pt}\goodbreak
\noindent therefore, if, under the conditions of Theorems~\ref{thm2} and \ref{t3},
\begin{gather}
\label{remar_4} n^{-1/2}b_n\to\infty \quad  \text{as} \  n\to
\infty,
\end{gather}
then MME is more efficient than LSE.
\end{remark}

In \cite{im13} (see also \cite{m14}), it is proved that if $F \in
D(\varLambda)$, then for any $\delta> 0$, $b_n=O(n^{\delta})$. From
this relation and Lemma~\ref{lemma1} it follows that (\ref{remar_4}) is not
satisfied for domains of maximum attraction $D(\varPhi_{\alpha})$ and
$D(\varLambda_{\alpha})$. In the case of domain $D(\varPsi_{\alpha})$,
condition (\ref{remar_4}) holds for $\alpha\in(0,2)$. For
example, assume that r.v.-s~$(\epsilon_j)$ are symmetrically
distributed on the interval $[-1,1]$ and
\begin{gather*}
1-F(1-h) = h^{\alpha}L(h) \quad\text{as} \ \ h\downarrow0, \ \alpha\in(0,2),
\end{gather*}
where $L(h)$ is an s.v.\ function at zero. Then $b_n=n^{1/\alpha}L_1(n)$,
where $L_1$ is an~s.v. at infinity function, and, under the conditions of
Theorems~\ref{thm2} and \ref{t3}, as $n\to\infty$,
\begin{gather*}
|\hat{\theta_i} - \theta_i| \overset{P}= O\bigl(
\bigl(n^{1/\alpha}L_1(n)\bigr)^{-1}\bigr) = o
\bigl(n^{-1/2}\bigr).
\end{gather*}

The next example also appears to be interesting.

\begin{exam}\label{exam3}
Let $(\epsilon_j) $ be uniformly distributed in $[-1, 1] $,
that is, $F(x)= \frac{x+1}{2} , \, x \in[-1, 1] $. It is
well known that $F\in D(\varPsi_1 )$, $a_n = 1, \, b_n =
\frac{n}{2} $. Then, under the conditions of Theorem~\ref{t3}, as $n
\rightarrow\infty$,
\[
\mathbb{P}\bigl( n(1-\hat{\varDelta}) < x \bigr)\rightarrow1 - \bigl[\mathbb{P}\{
\zeta_1+\zeta_2 > x \}\bigr]^q = 1-
(1+x)^q\exp(-qx) ,
\]
where $\zeta_1 , \zeta_2 $ are
i.i.d. r.v.-s, and $\mathbb{P}( \zeta_i
< x ) = 1- \exp(-x) ,\,
x>0 $.
\end{exam}
The following corollary is an immediate consequence of the Theorem~\ref{t3}.

\begin{cor}\label{cor1}

\label{c1} If for simple linear
regression \eqref{example_1}, conditions \eqref{x_j} and \eqref{x_j_V}
are satisfied, $k=q=2$, and
\begin{displaymath}
V= \lleft( %
\begin{matrix}%{cc}
1 & v_1 \\
1 & v_2
\end{matrix}
 \rright) ,
 \quad v_1
\neq v_2 ,
\end{displaymath}
then we have:
\begin{enumerate}
\item[{\rm(i)}]\quad\vspace*{-28pt}
\begin{align*}
&\hat{\varDelta} = \frac{1}{2} \max(R_{n1}, R_{n2}),\\
&\hat{\theta}_1 -\theta_1 = \frac{Q_{n2} - Q_{n1}}{v_2 - v_1} , \qquad
\hat{\theta}_0 -\theta_0 = \frac{Q_{n1} v_2 - Q_{n2} v_1}{v_2 -
v_1} ;
\end{align*}

\item[{\rm(ii)}] under assumptions $(A_1)$ and $(A_2)$, relation \eqref{f17}
holds for $q= 2 $,
and, as\break $n \rightarrow\infty$,
\[
2b_n (\hat{\theta}_1 -\theta_1 )
\stackrel{D} \longrightarrow \frac{\zeta_2 - \zeta'_2 - \zeta_1 + \zeta'_1 }{v_2 - v_1} ,
\]
\[
2b_n (\hat{\theta}_0 -\theta_0 )
\stackrel{D} \longrightarrow \frac{(\zeta_1 - \zeta'_1) v_2 - (\zeta_2 - \zeta'_2) v_1}{v_2 -
v_1} ,
\]
where the r.v.-s $\zeta_1 , \zeta'_1, \zeta_2, \zeta'_2$ are
independent and have d.f. $G$.
\end{enumerate}
\end{cor}

\begin{remark}\label{rem5}
The conditions of Theorem~\ref{t3} do not require \eqref{b_n_tend_infty}.
So it describes the asymptotic distribution of $\hat{\theta}$ even for
nonconsistent MME.
\end{remark}

%s3 ###
\section{Proofs of the main results}

Let us start with the following elementary lemma, where $Z_n(t)$,
$W_n(t)$, $R_n(t)$, and $Q_n(t)$ are determined by a sequence $t=\{t_1,
\ldots,t_n\}$ and are respectively the maximum, minimum, range,
and midrange of the sequence $t$.
\begin{lemma}\label{lem2}
Let $t_1, \ldots,t_n$ be any real numbers,
and
\begin{gather}
\label{alpha_n} \alpha_n = \underset{s \in R}\min\underset{1\leq j
\leq n}\max|t_j - s|.
\end{gather}
Then $\alpha_n = R_n(t)/2$; moreover, the minimum in \eqref{alpha_n} is
attained at the point $s = Q_n(t)$.
\end{lemma}
\begin{proof} Choose $s = Q_n(t)$. Then
\begin{gather*}
\underset{1\leq i\leq n} \max|t_i - s| = Z_n(t) -
Q_n(t) = Q_n(t) - W_n(t) =
\frac{1}{2}R_n(t).
\end{gather*}
If $s = Q_n(t) + \delta$, then, for $\delta> 0$,
\begin{gather*}
\underset{1\leq i\leq n} \max|t_i - s| = s -W_n(t) =
\frac{1}{2}R_n(t) + \delta,
\end{gather*}
and, for $\delta< 0$,
\begin{gather*}
\underset{1\leq i\leq n} \max|t_i - s| = Z_n(t) -s =
\frac{1}{2}R_n(t) - \delta,
\end{gather*}
that is, $s = Q_n(t)$ is the point of minimum.
\end{proof}

%{\it Proof of Statement 1.}
\begin{proof}[Proof of Statement~\ref{stat1}.]
We will use Lemma~\ref{lem2}:
\begin{equation*}
\hat{\varDelta} = \underset{\tau\in R^{q}}\min\underset{1\leq j\leq N}
\max\Bigg|\epsilon_j - \sum_{i=1}^q
(\tau_i - \theta_i)x_{ji}\Bigg| \leq
\\
\leq\underset{\tau_1 \in R^{q}}\min\underset{1\leq j\leq
N}\max\big|\epsilon_j - (\tau_1 - \theta_1)\big| =
\frac{1}{2}R_N
\end{equation*}
(we put $\tau_i =0$, $i\geq2$). The point (ii) of Statement~2 %\ref{stat2}
follows directly from Lemma~\ref{lem2}.
\end{proof}

%{\em Proof of Theorem 1. }
\begin{proof}[Proof of Theorem~\ref{thm1}.]
Using the notation
\[
d= (d_1 , \ldots, d_q ) , \quad d_i =
\tau_i -\theta_i , \ i={\overline{1, q}},
\]
and taking into account Eq.~(\ref{def_regres}), conditions
(\ref{x_j}) and (\ref{x_j_V}), we rewrite LPP
(\ref{LPP}) in the following form:
\begin{align}
\label{f21} \hat{\varDelta} &\,{=}\, \min_{\varDelta \in\mathcal{D}_1 } \varDelta ,\\
\mathcal{D}_1 &\,{=}\, \Biggl\{(d,\varDelta)\in\mathbb{R}^q \times\mathbb{R}_{+}: \sum_{i=1}^q d_i x_{ji} + \varDelta \geq\epsilon_j , -\sum_{i=1}^q d_i x_{ji} + \varDelta\geq-\epsilon_j , j={\overline{1, N}} \Biggr\}\notag\\
& \,{=}\,  \Biggl\{(d,\varDelta)\in\mathbb{R}^q \,{\times}\,\mathbb{R}_{+}\,{:}\,\sum_{i=1}^q d_i v_{li} \,{+}\, \varDelta \,{\geq}\, Z_{nl} , {-}\sum_{i=1}^q d_i v_{li} \,{+}\, \varDelta\geq-W_{nl} , l={\overline{1, k}} \Biggr\}.\notag
\end{align}
LPP dual to (\ref{f21}) has the form
%
%e5 ###
\begin{equation}
\label{f22} \max_{u \in\mathcal{D}^{*} }L_n^{*} (u) ,
\end{equation}
where $L_n^{*} (u) = \sum_{l=1}^k (u_l Z_{nl} - u'_l
W_{nl}) $, and the domain $\mathcal{D}^{*}$ is given by
(\ref{D_star}).

According to the basic duality theorem (\cite{mur}, Chap.~4),
\[
\hat{\varDelta} = \max_{u \in\mathcal{D}^{*} }L_n^{*} (u) .
\]
Hence, we obtain
\begin{align*}
b_n (\hat{\varDelta} -a_n ) &= \max_{u \in\mathcal{D}^{*} }
b_n \bigl(L_n^{*} (u) -a_n
\bigr) = \max_{u \in\mathcal{D}^{*} } g_n (u) ,\\
g_n (u) &= \sum_{l=1}^k
\bigl[u_l b_n (Z_{nl} -a_n ) +
u'_l b_n (-W_{nl}
-a_n ) \bigr] .
\end{align*}

Denote by $\varGamma^{*}$ the set of vertices of the domain
$\mathcal{D}^{*} $ and
\[
g_0 (u) = \sum_{l=1}^k
\bigl(u_l \zeta_l + u'_l
\zeta'_l \bigr) .
\]
Since the maximum in LPP (\ref{f22}) is attained at one of the
vertices $\varGamma^{*}$,
\[
\max_{u \in\mathcal{D}^{*} } g_n (u)= \max_{u \in\varGamma^{*} }
g_n (u) , \quad n\geq1 .
\]
Obviously, $\operatorname{card}( \varGamma^{*}) < \infty$. Thus, to prove
(\ref{theor_1_conv}), it suffices to prove that, as $n
\rightarrow\infty$
\[
\max_{u \in\varGamma^{*} } g_n (u) \stackrel{D} \longrightarrow
\max_{u \in\varGamma^{*} } g_0 (u)
\]
or
%
%e6 ###
\begin{equation}
\label{f23} \bigl(g_n (u), {u \in\varGamma^{*} }
\bigr)\stackrel{D} \longrightarrow \bigl(g_0 (u), {u \in
\varGamma^{*} } \bigr) .
\end{equation}

The Cramer--Wold argument (see, e.g., \S7 of the book
\cite{bil}) reduces (\ref{f23}) to the following relation: for
any $t_m \in R$ , as $n \rightarrow\infty$,
\[
\sum_{u^{(m)} \in\varGamma^{*} } g_n \bigl(u^{(m)}
\bigr)t_m \stackrel{D} \longrightarrow \sum
_{u^{(m)} \in\varGamma^{*} } g_0 \bigl(u^{(m)}
\bigr)t_m .
\]
The last convergence holds if for any $c_l , c'_l$, as $n
\rightarrow\infty$,
%
%e7 ###
\begin{equation}
\label{f24} \sum_{l=1}^k
\bigl[c_l (Z_{nl} - a_n) +
c'_l (-W_{nl} - a_n) \bigr]
\stackrel {D} \longrightarrow \sum_{l=1}^k
\bigl(c_l \zeta_l + c'_l
\zeta'_l \bigr) .
\end{equation}

Under the conditions of Theorem $1$,
\begin{align}
\label{f25} \zeta_{nl}&= b_n (Z_{nl}
-a_n ) \stackrel{D} \longrightarrow \zeta_l ,
\nonumber
\\
\zeta'_{nl} &= b_n (-W_{nl}
-a_n ) \stackrel{D} \longrightarrow\zeta'_l
, \quad l={\overline{1, k}} .
\end{align}
The vectors $(Z_{nl}, W_{nl})$, $l={\overline{1, k}}$,
are independent, and, on the other hand, $Z_{nl}$ and
$W_{nl}$ are asymptotically independent as $n \rightarrow\infty$
(\cite{llr}, p.~28). To obtain (\ref{f24}), it remains
to apply once more the Cramer--Wold argument. %$ \Box$
\end{proof}

%{\em Proof of Theorem 2 . }
\begin{proof}[Proof of Theorem~\ref{thm2}.]
Let $\hat{d} = (\hat{d}_1 ,
\ldots, \hat{d}_q ), \hat{\varDelta} $ be the
solution of LPP (\ref{f21}), and $\gamma_l = \sum_{i=1}^q
\hat{d}_i v_{li}$. Then, for any $l={\overline{1, k}}$,
\begin{align}
\label{f26} \gamma_l + \hat{\varDelta}&\geq Z_{nl} ,\nonumber\\
-\gamma_l + \hat{\varDelta}&\geq -W_{nl} .
\end{align}
Rewrite the asymptotic relation (\ref{f25}) and (\ref{theor_1_conv})
in the form
\begin{align}
\label{f27} &Z_{nl} = a_n + \frac{\zeta_{nl}}{b_n}, \qquad -W_{nl} = a_n + \frac{\zeta'_{nl}}{b_n},\\
&\zeta_{nl} \stackrel{D} \longrightarrow \zeta_l , \qquad
\zeta'_{nl} \stackrel{D} \longrightarrow
\zeta'_l,\nonumber
\end{align}
and
\begin{align}
\label{f28} &\hat{\varDelta} = a_n + \frac{\varDelta_{n}}{b_n},\\
&\varDelta_{n} \stackrel{D} \longrightarrow \varDelta_0 \quad
as \ n \rightarrow\infty.\nonumber
\end{align}
Combining (\ref{f26})--(\ref{f28}), we obtain, for $l={\overline{1, k}}$,
\begin{align*}
\gamma_{l} \geq Z_{nl}- \hat{\varDelta} &= \frac{\zeta_{nl}-\varDelta_{n}}{b_n} = O\bigl(b_n^{-1}\bigr),\\
\gamma_{l}\leq W_{nl}+ \hat{\varDelta} &= \frac{-\zeta'_{nl}+\varDelta_{n}}{b_n} =O\bigl(b_n^{-1}\bigr).
\end{align*}
Choose $l_1 , \ldots, l_q $ satisfying
(\ref{V_determ}). Then
\[
\sum_{i=1}^q \hat{d}_i
v_{l_j i} = \gamma_{l_j } = O\bigl(b_n^{-1}
\bigr), \quad j={\overline{1, q}},
\]
and by Cramer's rule,
\[
\hat{\theta}_i - \theta_i = \hat{d}_i =
\frac{\det\tilde{V}\gamma_{(i)}}{\det\tilde{V}} = O\bigl(b_n^{-1}\bigr),
\]
where the matrix $\tilde{V}\gamma_{(i)}$ is obtained from $\tilde{V}$
by replacement of the $i$th column by the column $(\gamma_{l_1}
, \ldots, \gamma_{l_q} )^T $. %$ \Box$
\end{proof}

%{\em Proof of Theorem $3$}.
\begin{proof}[Proof of Theorem \ref{t3}.]
(i) We have
\begin{align}
\label{f29} \varDelta&=  \min_{\tau\in R^q } \max_{1\leq l \leq q}\max_{j \in I_l}\left|y_j -\sum_{i=1}^{q}\tau_i v_{li}\right|\nonumber\\
& =  \min_{d \in R^q } \max_{1\leq l \leq q} \max_{j \in I_l}\left|\epsilon_j -\sum_{i=1}^{q} d_i v_{li}\right|.
\end{align}
By Lemma~\ref{lem2},
\[
\min_{s\in R } \max_{j \in I_l}|\epsilon_j
-s| = \frac{1}{2} R_{nl} \quad as \ s=Q_{nl} ,
\ l={\overline{1, q}} .
\]
Therefore, the minimum in $d$ is attained in (\ref{f29}) at the point
$\hat{d}$ being the solution of the system of linear equations
\[
\sum_{i=1}^{q} d_i
v_{li}= Q_{nl} , \quad l={\overline{1, q}} .
\]
Since the matrix $V$ is nonsingular, by Cramer's rule
\[
\hat{d}_i = \hat{\theta}_i - \theta_i =
\frac{\det VQ_{(i)}}{\det V} , \quad i={\overline{1, q}} .
\]
Obviously, for such a choice of $\hat{d}$, $\varDelta=
\frac{1}{2}\max_{1\leq l \leq q} R_{nl}$, thats is, we have obtained
formulae (\ref{f15}) and (\ref{f16}).

(ii) Using the asymptotic independence of r.v.-s $Z_n $ and
$W_n $, we derive the following statement.

\begin{lemma}\label{l3}
If r.v.-s $(\epsilon_j) $ satisfy conditions $(A_1)$, $(A_2 ) $,
then, as $n \rightarrow\infty$,
\begin{align}
\label{f30} b_n (R_n - 2a_{n} )&\stackrel{D} \longrightarrow \zeta +\zeta',\\
\label{f31} 2b_n Q_n &\stackrel{D} \longrightarrow\zeta-\zeta',
\end{align}
where $\zeta$ and $\zeta'$ are independent r.v.-s and have
d.f. $G$.
\end{lemma}

In fact, this lemma is contained in Theorem $2.9.2 $ of the book
\cite{gal} (see also Theorem $2.10 $ in \cite{m14}).

Equality (\ref{f17}) of Theorem $3$ follows immediately from
relation (\ref{f30}) of Lemma \ref{l3}.

Similarly, from the asymptotic relation (\ref{f31} ) and Eq.~(\ref
{f16}) we obtain (\ref{f18}) applying once more the Cramer--Wold
argument. %$ \Box$
\end{proof}

Remark\ref{rem3} follows directly from Theorem~\ref{t3}. Indeed, let $ k < q
$, and let there exist a nonsingular submatrix $\widetilde{V}
\subset V$,
\begin{gather*}
\widetilde{V} = \lleft(%
\begin{matrix}%{ccc}
v_{1 i_{1}} & \ldots& v_{1 i_{k}} \\
\ldots& \ldots& \ldots\\
v_{k i_{1}} & \ldots& v_{k i_{k}} \\
\end{matrix}
 \rright).
\end{gather*}
Choosing in LPP (21) from Theorem~\ref{thm1}, $d_i = 0 $ for all $ i \neq i_1 ,
i_2 , \ldots i_k$ (i.e., taking $\tau_i = \theta_i
$ for such indices $i$), we pass to the problem (29). It remains
to apply Eq.~(15) of Theorem~\ref{t3}.

\begin{remark}\label{rem6} Using the notation $\bar{\zeta} - \bar{\zeta'} =
(\zeta_1-
\zeta'_1 , \ldots, \zeta_q- \zeta'_q)^{T} $, the
coordinatewise relation \eqref{f18} of Theorem $3$ can be
rewritten in the equivalent vector form
%
%e8 ###
\begin{equation}
\label{f32} 2b_n(\hat{\theta} - \theta) \stackrel{D}
\longrightarrow V^{-1} \bigl(\bar{\zeta} - \bar{\zeta'}
\bigr) \quad as \ n \rightarrow \infty.
\end{equation}
If $\operatorname{Var}\zeta= \sigma_G^2 $ of r.v. $\zeta$ having d.f.G exists,
then the covariance matrix of the limiting distribution in \eqref{f32}
is $C_G = 2\sigma_G^2 (V^T V )^{-1}$.
\end{remark}

%\bibliography{bib/bibliogr}
%

\end{document}